\documentclass[a4paper, 11pt]{amsart}

\usepackage{amsthm, amsfonts, amsmath, amssymb}
\usepackage{enumerate}			
\usepackage{subfigure} 	   	
\usepackage{tikz}           
\usepackage{ifthen}					
\usepackage{amsaddr}

\newtheorem{thm}{Theorem}[section]
\newtheorem{lem}[thm]{Lemma}

\newtheorem{cor}[thm]{Corollary}

\theoremstyle{definition}\newtheorem{definition}[thm]{Definition}
\newtheorem{claim}[thm]{Claim}

\newtheorem{example}[thm]{Example}
\newtheorem{conjecture}[thm]{Conjecture}
\newtheorem*{cor:InductiveConstruction_TreesAndGraph}{Corollary \ref{cor:InductiveConstruction_TreesAndGraph}}
\newtheorem*{cor:RealisationIn2Dimensions}{Corollary \ref{cor:RealisationIn2Dimensions}}

\theoremstyle{definition}

\tikzset{vertex/.style={circle,fill=black,minimum size=0.4pt,inner sep=1.8pt}}
\tikzset{emptyvertex/.style={shape=circle, inner sep = 1.8pt, minimum size=0.4pt, draw}}

 \def\G{{\mathcal{G}}}
\renewcommand{\H}{{\mathcal{H}}}

\title{Block-and-hole graphs:\\ Constructibility and $(3,0)$-sparsity}

\author{Bryan Gin--ge Chen}
\email{bryangingechen@gmail.com}

\author{James Cruickshank}
\address[James Cruickshank]{School of Mathematical and Statistical Sciences, University of Galway, Galway, Ireland.}
\email{james.cruickshank@universityofgalway.ie}
\author{Derek Kitson}
\address[Derek Kitson]{Department of Mathematics and Computer Studies, Mary Immaculate College, Thurles, Co.~Tipperary, Ireland.
}
\email{Derek.Kitson@mic.ul.ie}

\begin{document}
\begin{abstract}
We show that minimally 3-rigid block-and-hole graphs, with one block or one hole, are characterised as those which are constructible from $K_3$ by vertex splitting, and also, as those having associated looped face graphs which are $(3,0)$-tight. This latter property can be verified in polynomial time by a form of pebble game algorithm. We also indicate connections to the rigidity properties of polyhedral surfaces known as {\em origami} and to graph rigidity in $\ell_p^3$ for $p\not=2$.       
\end{abstract}

\maketitle

\section{Introduction}
A finite simple graph is {\em $3$-rigid} if it forms the structure graph for an infinitesimally rigid bar-and-joint framework in Euclidean $3$-space. If, in addition, the removal of any edge from the graph results in a subgraph which is not $3$-rigid then the graph is {\em minimally} $3$-rigid. 
A {\em block-and-hole graph} is obtained by first triangulating a sphere, then removing the interiors of some triangulated discs to create holes, and finally adjoining minimally $3$-rigid graphs to the boundaries of some of the resulting holes to create blocks.
It is well known that a graph obtained from a triangulation of a sphere is minimally $3$-rigid, see for example \cite{gluck}. Whiteley (\cite[Theorem 4.2]{whiteley2}) showed that a block-and-hole graph with a single block and a single hole, and common boundary length $k$, is minimally 3-rigid if and only if the removal of any $k-1$ vertices does not disconnect the two boundary cycles. In \cite{frw}, it is shown that switching the blocks and holes in a block-and-hole graph preserves minimal $3$-rigidity. The main theorem of \cite{ckp} characterizes the
minimally 3-rigid block-and-hole graphs with a single block and finitely
many holes (or vice versa) as those which are $(3,6)$-tight.  Moreover, Whiteley's result is generalised in the form of girth inequalities  and a constructive characterisation is obtained which uses the well-known graph move known as {\em vertex-splitting} together with a form of graph fusion known as {\em isostatic  substitution}. In recent work, Jord\'{a}n (\cite{jordan}) has obtained a rank formula for the $3$-dimensional rigidity matroid of a block-and-hole graph with a single block.

In Section \ref{s:split}, we improve the constructive characterisation obtained in \cite{ckp} by circumventing the need for isostatic substitution. The main result, Theorem \ref{thm:MainImproved}, characterises minimally $3$-rigid block-and-hole graphs, with one block or one hole, as those for which an associated discus-and-hole graph is constructible from the complete graph $K_3$ by a sequence of vertex splitting operations. In Section \ref{s:(3,0)}, we present new characterisations of minimal $3$-rigidity for block-and-hole graphs which are expressed in terms of associated multigraphs, referred to as {\em looped face graphs}. The main result, Theorem \ref{thm:(3,0)}, shows that a block-and-hole graph, with one block or one hole, is  minimally $3$-rigid if and only if its associated looped face graphs are $(3,0)$-tight.
This latter property has a significant advantage that, unlike $(3,6)$-tightness, it is verifiable in polynomial time by a  pebble game algorithm (see \cite{leestreinu}).
In Section \ref{s:applications}, we  apply our constructive characterisation to show that $(3,6)$-tight discus-and-hole graphs are independent in any 3-dimensional real normed linear space which is smooth and strictly convex. Conjectures are made on the minimal rigidity of block-and-hole graphs in the normed space $\ell_p^3$, for $p\in[1,\infty]$, $p\not=2$, and on the global rigidity of discus-and-hole graphs in the Euclidean space $\mathbb{R}^3$. Finally, we indicate connections between the rigidity properties of block-and-hole graphs and the rigidity properties of polyhedral surfaces known as {\em origami}.
   
In what follows, we use the definition and notation of block-and-hole graphs
and related terminology, from  \cite{ckp}. 
Let $S = (V, E)$ be the graph of a triangulated sphere (i.e.~a maximal planar graph). Let $c$ be a simple cycle
in $S$ of length four or more. Then $c$ determines two complementary planar subgraphs of $S$, each with a single non-triangular face bordered by the edges of $c$. Such a subgraph $D$ is referred to as a {\em simplicial disc} of $S$ with {\em boundary cycle} $\partial D:=c$. 
A collection of simplicial discs is {\em internally-disjoint} if their respective triangular faces are pairwise disjoint.

\begin{definition}
A {\em face graph} is a simple graph, $G$, which is obtained from the graph of a triangulated sphere, $S$, by,
\begin{enumerate}[(i)]
    \item choosing a collection of internally disjoint simplicial discs in $S$,
    \item removing the vertices and edges of each of these simplicial discs except for the vertices and edges which lie in the boundary cycles of the simplicial discs,
    \item labelling the non-triangular faces of the resulting planar graph by either $B$ or $H$.
\end{enumerate}
\end{definition}
A $BH$ edge in a face graph $G$ is an edge which lies in the boundary of a $B$-labelled face and a $H$-labelled face. A $TT$ edge in $G$ is an edge which lies in the boundary of two triangular faces.

\begin{example}
Figure \ref{f:facegraph} illustrates the three steps in the construction of a face graph beginning  on the left hand side with a maximal planar graph. Two internally disjoint simplicial discs are chosen with boundary cycles indicated in red and blue. Non-boundary vertices and edges of the chosen simplicial discs are removed (centre) and finally non-triangular faces are labelled by either $B$ or $H$ (right).     
\end{example}

\begin{figure}
\label{f:facegraph}
\definecolor{uququq}{rgb}{0.25,0.25,0.25}
\definecolor{zzttqq}{rgb}{0.6,0.2,0}
\definecolor{qqqqff}{rgb}{0,0,1}

\begin{tikzpicture}[line cap=round,line join=round,x=0.7cm,y=0.7cm]
\begin{scope}[shift={(0,3)},scale=0.5]
\clip(-7.1,-2) rectangle (17.52,8.88);
\draw [color=zzttqq] (-1.74,-0.08)-- (2.8,1.1);
\draw [color=zzttqq] (2.8,1.1)-- (5.07,3.85);
\draw [color=zzttqq] (5.07,3.85)-- (2.4,6.78);
\draw [color=zzttqq] (2.4,6.78)-- (-1.74,7.78);
\draw [color=zzttqq] (-1.74,7.78)-- (-3.01,3.85);
\draw [color=zzttqq] (-3.01,3.85)-- (-1.74,-0.08);
\draw  (-1.74,7.78)-- (-1.66,5.22);
\draw (-1.66,5.22)-- (2.66,5.34);
\draw [color=blue] (2.66,5.34)-- (2.4,6.78);
\draw (2.66,5.34)-- (5.07,3.85);
\draw [color=blue] (2.66,5.34)-- (0.56,1.84);
\draw  (0.56,1.84)-- (2.8,1.1);
\draw (0.56,1.84)-- (-1.74,-0.08);
\draw [color=blue] (0.56,1.84)-- (-1.66,5.22);
\draw (-1.66,5.22)-- (-3.01,3.85);

\draw  (-1.74,-0.08) to [out=-30,in=-90] (5.07,3.85);
\draw  (5.07,3.85) to [out=90,in=30] (-1.74,7.78);
\draw  (-1.74,7.78) to [out=180,in=180] (-1.74,-0.08);

\draw  [color=blue] (-1.66,5.22) -- (2.4,6.78);
\draw  (-1.66,5.22) -- (-1.74,-0.08);
\draw  (2.66,5.34) -- (2.8,1.1);

\begin{scriptsize}
\fill [color=uququq] (-1.74,-0.08) circle (2.5pt);
\fill [color=uququq] (2.8,1.1) circle (2.5pt);
\fill [color=uququq] (5.07,3.85) circle (2.5pt);
\fill [color=uququq] (2.4,6.78) circle (2.5pt);
\fill [color=uququq] (-1.74,7.78) circle (2.5pt);
\fill [color=uququq] (-3.01,3.85) circle (2.5pt);
\fill [color=uququq] (-1.66,5.22) circle (2.5pt);
\fill [color=uququq] (2.66,5.34) circle (2.5pt);
\fill [color=uququq] (0.56,1.84) circle (2.5pt);
\end{scriptsize}

\end{scope}

\begin{scope}[shift={(5.5,3)},scale=0.5]
\clip(-7.1,-2) rectangle (17.52,8.88);
\draw [color=zzttqq] (-1.74,-0.08)-- (2.8,1.1);
\draw [color=zzttqq] (2.8,1.1)-- (5.07,3.85);
\draw [color=zzttqq] (5.07,3.85)-- (2.4,6.78);
\draw [color=zzttqq] (2.4,6.78)-- (-1.74,7.78);
\draw [color=zzttqq] (-1.74,7.78)-- (-3.01,3.85);
\draw [color=zzttqq] (-3.01,3.85)-- (-1.74,-0.08);
\draw (-1.74,7.78)-- (-1.66,5.22);
\draw (2.66,5.34)-- (2.4,6.78);
\draw (2.66,5.34)-- (5.07,3.85);
\draw (2.66,5.34)-- (0.56,1.84);
\draw (0.56,1.84)-- (2.8,1.1);
\draw (0.56,1.84)-- (-1.74,-0.08);
\draw (0.56,1.84)-- (-1.66,5.22);
\draw (-1.66,5.22)-- (-3.01,3.85);

\draw  (-1.66,5.22) -- (2.4,6.78);
\draw  (-1.66,5.22) -- (-1.74,-0.08);
\draw  (2.66,5.34) -- (2.8,1.1);

\begin{scriptsize}
\fill [color=uququq] (-1.74,-0.08) circle (2.5pt);
\fill [color=uququq] (2.8,1.1) circle (2.5pt);
\fill [color=uququq] (5.07,3.85) circle (2.5pt);
\fill [color=uququq] (2.4,6.78) circle (2.5pt);
\fill [color=uququq] (-1.74,7.78) circle (2.5pt);
\fill [color=uququq] (-3.01,3.85) circle (2.5pt);
\fill [color=uququq] (-1.66,5.22) circle (2.5pt);
\fill [color=uququq] (2.66,5.34) circle (2.5pt);
\fill [color=uququq] (0.56,1.84) circle (2.5pt);
\end{scriptsize}

\end{scope}

\begin{scope}[shift={(11,3)},scale=0.5]
\clip(-7.1,-2) rectangle (17.52,8.88);
\draw [color=zzttqq] (-1.74,-0.08)-- (2.8,1.1);
\draw [color=zzttqq] (2.8,1.1)-- (5.07,3.85);
\draw [color=zzttqq] (5.07,3.85)-- (2.4,6.78);
\draw [color=zzttqq] (2.4,6.78)-- (-1.74,7.78);
\draw [color=zzttqq] (-1.74,7.78)-- (-3.01,3.85);
\draw [color=zzttqq] (-3.01,3.85)-- (-1.74,-0.08);
\draw (-1.74,7.78)-- (-1.66,5.22);
\draw (2.66,5.34)-- (2.4,6.78);
\draw (2.66,5.34)-- (5.07,3.85);
\draw (2.66,5.34)-- (0.56,1.84);
\draw (0.56,1.84)-- (2.8,1.1);
\draw (0.56,1.84)-- (-1.74,-0.08);
\draw (0.56,1.84)-- (-1.66,5.22);
\draw (-1.66,5.22)-- (-3.01,3.85);

\draw  (-1.66,5.22) -- (2.4,6.78);
\draw  (-1.66,5.22) -- (-1.74,-0.08);
\draw  (2.66,5.34) -- (2.8,1.1);

\draw (-0.5,5.2) node[anchor=north west] {$ H $};
\draw (-4.1,6.84) node[anchor=north west] {$ B $};

\begin{scriptsize}
\fill [color=uququq] (-1.74,-0.08) circle (2.5pt);
\fill [color=uququq] (2.8,1.1) circle (2.5pt);
\fill [color=uququq] (5.07,3.85) circle (2.5pt);
\fill [color=uququq] (2.4,6.78) circle (2.5pt);
\fill [color=uququq] (-1.74,7.78) circle (2.5pt);
\fill [color=uququq] (-3.01,3.85) circle (2.5pt);
\fill [color=uququq] (-1.66,5.22) circle (2.5pt);
\fill [color=uququq] (2.66,5.34) circle (2.5pt);
\fill [color=uququq] (0.56,1.84) circle (2.5pt);
\end{scriptsize}

\end{scope}

\end{tikzpicture}
\begin{caption}{Constructing a face graph. }
    \label{fig:facegraph}
\end{caption}
\end{figure}

\begin{definition}
A {\em block-and-hole graph}  
is a simple graph of the form $\hat{G}=G\cup \hat{B}_1\cup\cdots\cup \hat{B}_m$ where,
\begin{enumerate}[(i)]
\item $G$ is a face graph with $m$ $B$-labelled faces $B_1,\ldots,B_m$, 
\item $\hat{B}_1,\ldots,\hat{B}_m$ are minimally $3$-rigid graphs,
\item $G\cap \hat{B}_i = \partial B_i$, for each $i=1,\ldots,m$.
\end{enumerate}
\end{definition}
We refer to the minimally $3$-rigid graphs $\hat{B}_1,\ldots,\hat{B}_m$ as {\em blocks} and the $H$-labelled faces of $G$ as {\em holes}.

For each $B$-labelled face $B_i$ we can construct a block 
 $B_i^\dagger$ with,
\[
V(B_i^\dagger) = V(\partial B_i) \cup\{x_i,y_i\},
\quad E(B_i^\dagger) =E(\partial B_i) \cup \{(v,x_i),(v,y_i):v \in  V(\partial B_i)\}.
\]
The block $B_i^\dagger$ is referred to as a {\em simplicial discus} with {\em poles} at $x_i$ and $y_i$. 
The resulting block-and-hole graph $G^\dagger:=G\cup B_1^\dagger\cup\cdots\cup B_m^\dagger$ is  referred to as the {\em discus-and-hole graph} for $G$.

Let $f(J)$ denote the \emph{freedom number} $3|V(J)|-|E(J)|$ of a graph $J$. 
A simple graph $J$ is said to be  {\em $(3,6)$-sparse} if $f(J')\geq 6$ for any subgraph $J'$ containing at least two edges. The graph $J$ is {\em $(3,6)$-tight} if it is $(3,6)$-sparse and $f(J)=6$. We denote by $\mathcal{G}(m, n)$ the set of face graphs with $m$ $B$-labelled faces and $n$ $H$-labelled faces for which the discus-and-hole graph $G^\dagger$ is $(3, 6)$-tight.

We will make reference to the following theorem which is proved in \cite{ckp}.

\begin{thm}\label{t:ckpthm}  
 Let $\hat{G}$ be a block-and-hole graph with a single block and finitely many holes, or, a single hole and finitely many blocks. 
 Then the following statements are equivalent.
  \begin{enumerate}[(i)]
\item $\hat{G}$ is minimally $3$-rigid.
\item $\hat{G}$ is $(3,6)$-tight.
\item $\hat{G}$ is constructible from $K_3$ by vertex splitting and isostatic  substitution.
\item $\hat{G}$ satisfies the girth inequalities. 
\end{enumerate}
\end{thm}

\section{Vertex splitting}
\label{s:split}
Let $J$ be a simple graph and let $v$ be a vertex of $J$ with adjacent vertices $v_1,v_2,\ldots, v_n$, $n\geq 2$.
Construct a new graph $\tilde{J}$ from $J$ by,
\begin{enumerate}[(i)]
\item removing the vertex $v$ and its incident edges from $J$,
\item adjoining two new vertices $w_1,w_2$,
\item adjoining the edge $w_1v_j$ or the edge $w_2v_j$ for each 
$j=3,4,\ldots,n$,
\item adjoining the five edges $v_1w_1,v_2w_1$, $v_1w_2,v_2w_2$ and $w_1w_2$.
\end{enumerate}
The graph $\tilde{J}$ is said to be obtained from $J$ by  {\em (3-dimensional) vertex splitting}. See Figure \ref{fig:vsplit} for an illustration.

\begin{figure}[ht]
	\begin{tikzpicture}
		\node[vertex] (0) at (-6,0.5) {};
		\node[vertex] (1) at (-6.8,0.5) {};
		\node[vertex] (2) at (-5.2,0.5) {};
		
		\node (n1) at (-6.3,1.3) {};
		\node (n2) at (-5.7,1.4) {};
		\node (n3) at (-5.7,-0.3) {};
		
		\draw (0)edge(1);
		\draw (0)edge(2);	

		\draw (0)edge(n1);
		\draw (0)edge(n2);	
		\draw (0)edge(n3);	
			
		\draw[->] (-3.5,0.5) -- (-2.5,0.5);
		
		\node (m1) at (-0.55,1.3) {};
		\node (m2) at (0.55,1.4) {};
		\node (m3) at (0.55,-0.3) {};

		\node[vertex] (0') at (0,1) {};
		\node[vertex] (0'') at (0,0) {};
		\node[vertex] (1') at (-0.8,0.5) {};
		\node[vertex] (2') at (0.8,0.5) {};
		
		\draw (0')edge(1');
		\draw (0')edge(2');
		\draw (0'')edge(1');
		\draw (0'')edge(2');
		\draw (0')edge(0'');
		
		\draw (0')edge(m1);
		\draw (0')edge(m2);	
		\draw (0'')edge(m3);	
	\end{tikzpicture}
	\caption{A vertex splitting operation.}
	\label{fig:vsplit}
\end{figure}
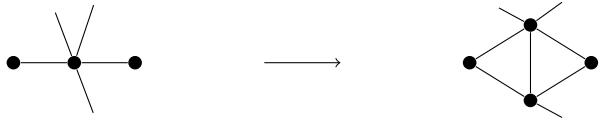

In this section we show that a block-and-hole graph with a single block, or a single hole, is minimally $3$-rigid if and only if the corresponding discus-and-hole graph is constructible from $K_3$ by vertex splitting. For more on vertex splitting and rigid graphs see \cite{whi-vertex} for example.

\subsection{Critical separating cycles}
Let $G$ be a face graph with exactly one $B$-labelled face and any number of $H$-labelled faces. 
Fix a planar realisation of $G$ such that the unbounded face is $B$-labelled. Let $c$ be a simple cycle in $G$.
Define $G_1$ to be the face graph obtained from $G$ and $c$ by,
\begin{enumerate}[(i)]
\item removing all edges and vertices interior to $c$, and,
\item if $|c|\geq4$, viewing the edges of $c$ as the boundary of a new face with label  $H$. 
\end{enumerate}
Define $G_2$ to be the face graph obtained from $G$ and $c$ by,
\begin{enumerate}[(i)]
\item  removing all edges and vertices which are exterior to $c$, and, 
\item if $|c|\geq4$, viewing the edges of $c$ as the boundary of a new face with label  $B$. 
\end{enumerate}
We refer to $G_1$ and $G_2$ respectively as the {\em external} and {\em internal face graphs} associated with $c$. 
See Figure \ref{f:ExtInt} for an illustration.

\begin{figure}
\definecolor{uququq}{rgb}{0.25,0.25,0.25}
\definecolor{zzttqq}{rgb}{0.6,0.2,0}
\definecolor{qqqqff}{rgb}{0,0,1}

\begin{tikzpicture}[line cap=round,line join=round,x=0.7cm,y=0.7cm]

\begin{scope}[shift={(4,3)},scale=0.5]
\clip(-6.1,-0.85) rectangle (17.52,8.88);
\draw [color=zzttqq] (-1.74,-0.08)-- (2.8,1.1);
\draw [color=zzttqq] (2.8,1.1)-- (5.07,3.85);
\draw [color=zzttqq] (5.07,3.85)-- (2.4,6.78);
\draw  (2.4,6.78)-- (-1.74,7.78);
\draw  (-1.74,7.78)-- (-3.01,3.85);
\draw [color=zzttqq] (-3.01,3.85)-- (-1.7,-0.08);
\draw  (-1.74,7.78)-- (-1.66,5.22);
\draw  (2.66,5.34)-- (2.4,6.78);
\draw  (2.66,5.34)-- (5.07,3.85);
\draw  (2.66,5.34)-- (0.56,1.84);
\draw  (0.56,1.84)-- (2.8,1.1);
\draw  (0.56,1.84)-- (-1.74,-0.08);
\draw [color=zzttqq] (-1.66,5.22)-- (-3.01,3.85);

\draw [color=zzttqq] (-1.66,5.22) -- (2.4,6.78);
\draw  (2.66,5.34) -- (2.8,1.1);

\draw (-1.5,4) node[anchor=north west] {$ H $};
\draw (-4.1,6.84) node[anchor=north west] {$ B $};

\begin{scriptsize}
\fill [color=uququq] (-1.74,-0.08) circle (2.5pt);
\fill [color=uququq] (2.8,1.1) circle (2.5pt);
\fill [color=uququq] (5.07,3.85) circle (2.5pt);
\fill [color=uququq] (2.4,6.78) circle (2.5pt);
\fill [color=uququq] (-1.74,7.78) circle (2.5pt);
\fill [color=uququq] (-3.01,3.85) circle (2.5pt);
\fill [color=uququq] (-1.66,5.22) circle (2.5pt);
\fill [color=uququq] (2.66,5.34) circle (2.5pt);
\fill [color=uququq] (0.56,1.84) circle (2.5pt);
\end{scriptsize}
\end{scope}

\begin{scope}[shift={(9.5,3)},scale=0.5]
\clip(-12.1,-0.85) rectangle (17.52,8.88);
\draw [color=zzttqq] (-1.74,-0.08)-- (2.8,1.1);
\draw [color=zzttqq] (2.8,1.1)-- (5.07,3.85);
\draw [color=zzttqq] (5.07,3.85)-- (2.4,6.78);
\draw [color=zzttqq] (2.4,6.78)-- (-1.74,7.78);
\draw [color=zzttqq] (-1.74,7.78)-- (-3.01,3.85);
\draw [color=zzttqq] (-3.01,3.85)-- (-1.7,-0.08);
\draw [color=zzttqq] (-1.74,7.78)-- (-1.66,5.22);
\draw [color=zzttqq] (-1.66,5.22)-- (-3.01,3.85);

\draw [color=zzttqq] (-1.66,5.22) -- (2.4,6.78);

\draw (-0.5,4) node[anchor=north west] {$ H $};
\draw (-4.1,6.84) node[anchor=north west] {$ B $};

\begin{scriptsize}
\fill [color=uququq] (-1.74,-0.08) circle (2.5pt);
\fill [color=uququq] (2.8,1.1) circle (2.5pt);
\fill [color=uququq] (5.07,3.85) circle (2.5pt);
\fill [color=uququq] (2.4,6.78) circle (2.5pt);
\fill [color=uququq] (-1.74,7.78) circle (2.5pt);
\fill [color=uququq] (-3.01,3.85) circle (2.5pt);
\fill [color=uququq] (-1.66,5.22) circle (2.5pt);
\end{scriptsize}
\end{scope}

\begin{scope}[shift={(15,3)},scale=0.5]
\clip(-12.1,-0.85) rectangle (17.52,8.88);
\draw [color=zzttqq] (-1.74,-0.08)-- (2.8,1.1);
\draw [color=zzttqq] (2.8,1.1)-- (5.07,3.85);
\draw [color=zzttqq] (5.07,3.85)-- (2.4,6.78);
\draw [color=zzttqq] (-3.01,3.85)-- (-1.7,-0.08);
\draw [color=zzttqq] (2.66,5.34)-- (2.4,6.78);
\draw [color=zzttqq] (2.66,5.34)-- (5.07,3.85);
\draw [color=zzttqq] (2.66,5.34)-- (0.56,1.84);
\draw [color=zzttqq] (0.56,1.84)-- (2.8,1.1);
\draw [color=zzttqq] (0.56,1.84)-- (-1.74,-0.08);
\draw [color=zzttqq] (-1.66,5.22)-- (-3.01,3.85);

\draw [color=zzttqq] (-1.66,5.22) -- (2.4,6.78);
\draw  [color=zzttqq] (2.66,5.34) -- (2.8,1.1);

\draw (-1.5,4) node[anchor=north west] {$ H $};
\draw (-3.1,6.84) node[anchor=north west] {$ B $};

\begin{scriptsize}
\fill [color=uququq] (-1.74,-0.08) circle (2.5pt);
\fill [color=uququq] (2.8,1.1) circle (2.5pt);
\fill [color=uququq] (5.07,3.85) circle (2.5pt);
\fill [color=uququq] (2.4,6.78) circle (2.5pt);
\fill [color=uququq] (-3.01,3.85) circle (2.5pt);
\fill [color=uququq] (-1.66,5.22) circle (2.5pt);
\fill [color=uququq] (2.66,5.34) circle (2.5pt);
\fill [color=uququq] (0.56,1.84) circle (2.5pt);
\end{scriptsize}
\end{scope}

\end{tikzpicture}
\begin{caption}{Left: A cycle $c$ (indicated in red)  in a face graph with one $B$-labelled face. Centre: The associated external face graph $G_1$.
Right: The associated internal face graph $G_2$.}
    \label{f:ExtInt}
\end{caption}
\end{figure}
Note that in the case where $|c|=3$, the internal face graph $G_2$ has no $B$-labelled face.
We denote by $Ext_G(c)$, or simply $Ext(c)$ when the context is clear, the discus-and-hole graph for the external face graph $G_1$. Note that $Ext(c)$ is a block-and-hole graph with a single block and so, by Theorem \ref{t:ckpthm}, $Ext(c)$ is $(3,6)$-tight if and only if it is minimally $3$-rigid.

\begin{definition}
A {\em critical separating cycle} for a face graph $G\in \mathcal{G}(1,n)$ 
is a simple cycle $c$ in $G$ with the property that the external discus-and-hole graph $Ext(c)$ is $(3,6)$-tight. 
\end{definition}

We will require the following lemma which is adapted from the proof of \cite[Proposition 22]{ckp}.

\begin{lem}
\label{l:commonH}
Let  $G\in \G(1,n)$ and let $v$ and $w$ be distinct vertices in $\partial B$ which are not joined by a $BH$ edge in $G$. If $v$ and $w$ lie in a common $H$-labelled face  then $G$ contains a non-facial critical separating cycle.
\end{lem}

\proof
Suppose there exists a $H$-labelled face in $G$ which contains the vertices $v$ and $w$. 
The boundary of this $H$-labelled face is composed of two edge-disjoint paths $\pi_1$ and $\pi_2$ joining $v$ to $w$.
Let $c_1$ be the simple cycle in $\partial B\cup \partial H$ which contains the path $\pi_1$ and has the property that $Ext(c_1)$ does not contain the path $\pi_2$. Similarly, let $c_2$ be the simple cycle in $\partial B\cup \partial H$ which contains the path $\pi_2$ and has the property that $Ext(c_2)$ does not contain the path $\pi_1$.
See Figure \ref{f:commonH} for an illustration.
Note that 
$Ext(c_1)\cap Ext(c_2) = B^\dagger$. Thus, 
\[f(G^\dagger)=f(Ext(c_1))+f(Ext(c_2))-f(B^\dagger).\]
Since $f(G^\dagger)=f(B^\dagger)=6$, it follows that
$f(Ext(c_1))=f(Ext(c_2))=6$.
Hence $Ext(c_1)$ and $Ext(c_2)$ are both $(3,6)$-tight and so $c_1$ and $c_2$ are non-facial critical separating cycles for $G$.

\endproof

\begin{figure}
\definecolor{uququq}{rgb}{0.25,0.25,0.25}
\definecolor{zzttqq}{rgb}{0.6,0.2,0}
\definecolor{qqqqff}{rgb}{0,0,1}

\begin{tikzpicture}[line cap=round,line join=round,x=0.7cm,y=0.7cm]
\begin{scope}[shift={(1,3)},scale=0.6]
\clip(-8.1,-0.85) rectangle (17.52,8.88);
\draw  (-1.74,-0.08)-- (2.8,1.1);
\draw  (2.8,1.1)-- (5.07,3.85);
\draw  (5.07,3.85)-- (2.4,6.78);
\draw  (2.4,6.78)-- (-1.74,7.78);
\draw  (-1.74,7.78)-- (-3.01,3.85);
\draw [color=zzttqq] (-3.01,3.85)-- (-1.74,-0.08);
\draw  (-1.74,7.78)-- (-1.66,5.22);
\draw [color=blue] (2.66,5.34)-- (2.4,6.78);
\draw (2.66,5.34)-- (5.07,3.85);
\draw [color=blue] (2.66,5.34)-- (0.56,1.84);
\draw (0.56,1.84)-- (2.8,1.1);
\draw [color=blue] (0.56,1.84)-- (-1.74,-0.08);
\draw [color=zzttqq] (-1.66,5.22)-- (-3.01,3.85);

\draw [color=zzttqq] (-1.66,5.22) -- (2.4,6.78);
\draw  (2.66,5.34) -- (2.8,1.1);

\draw (-1.5,4) node[anchor=north west] {$ H $};
\draw (-4.1,6.84) node[anchor=north west] {$ B $};
\draw (2.1,7.7) node[anchor=north west] {$ v $};
\draw (-3,0.25) node[anchor=north west] {$ w $};

\begin{scriptsize}
\fill [color=uququq] (-1.74,-0.08) circle (2.5pt);
\fill [color=uququq] (2.8,1.1) circle (2.5pt);
\fill [color=uququq] (5.07,3.85) circle (2.5pt);
\fill [color=uququq] (2.4,6.78) circle (2.5pt);
\fill [color=uququq] (-1.74,7.78) circle (2.5pt);
\fill [color=uququq] (-3.01,3.85) circle (2.5pt);
\fill [color=uququq] (-1.66,5.22) circle (2.5pt);
\fill [color=uququq] (2.66,5.34) circle (2.5pt);
\fill [color=uququq] (0.56,1.84) circle (2.5pt);
\end{scriptsize}
\end{scope}

\begin{scope}[shift={(9,3)},scale=0.6]
\clip(-12.1,-0.85) rectangle (17.52,8.88);
\draw [color=zzttqq] (-1.74,-0.08)-- (2.8,1.1);
\draw [color=zzttqq] (2.8,1.1)-- (5.07,3.85);
\draw [color=zzttqq] (5.07,3.85)-- (2.4,6.78);
\draw [color=blue] (2.4,6.78)-- (-1.74,7.78);
\draw [color=blue] (-1.74,7.78)-- (-3.01,3.85);
\draw [color=zzttqq] (-2.95,3.85)-- (-1.7,-0.08);
\draw [color=blue] (-3.1,3.85)-- (-1.85,-0.08);
\draw  (-1.74,7.78)-- (-1.66,5.22);
\draw [color=blue] (2.66,5.34)-- (2.4,6.78);
\draw (2.66,5.34)-- (5.07,3.85);
\draw [color=blue] (2.66,5.34)-- (0.56,1.84);
\draw (0.56,1.84)-- (2.8,1.1);
\draw [color=blue] (0.56,1.84)-- (-1.74,-0.08);
\draw [color=zzttqq] (-1.66,5.22)-- (-3.01,3.85);

\draw [color=zzttqq] (-1.66,5.22) -- (2.4,6.78);
\draw  (2.66,5.34) -- (2.8,1.1);

\draw (-1.5,4) node[anchor=north west] {$ H $};
\draw (-4.1,6.84) node[anchor=north west] {$ B $};
\draw (2.1,7.7) node[anchor=north west] {$ v $};
\draw (-3,0.25) node[anchor=north west] {$ w $};

\begin{scriptsize}
\fill [color=uququq] (-1.74,-0.08) circle (2.5pt);
\fill [color=uququq] (2.8,1.1) circle (2.5pt);
\fill [color=uququq] (5.07,3.85) circle (2.5pt);
\fill [color=uququq] (2.4,6.78) circle (2.5pt);
\fill [color=uququq] (-1.74,7.78) circle (2.5pt);
\fill [color=uququq] (-3.01,3.85) circle (2.5pt);
\fill [color=uququq] (-1.66,5.22) circle (2.5pt);
\fill [color=uququq] (2.66,5.34) circle (2.5pt);
\fill [color=uququq] (0.56,1.84) circle (2.5pt);
\end{scriptsize}
\end{scope}

\end{tikzpicture}
\begin{caption}{An illustration of the proof of Lemma \ref{l:commonH}. The edge-disjoint paths $\pi_1$ and $\pi_2$ are indicated in red and blue on the left. The cycles $c_1$ and $c_2$ are indicated in red and blue on the right. }
    \label{f:commonH}
\end{caption}
\end{figure}

We will require the following result, known as the ``hole-filling" lemma. In the statement of the lemma, $int(c)$ denotes the subgraph of $G$ spanned by edges which lie inside the cycle $c$. 

\begin{lem}[{\cite[Lemma 26]{ckp}}]
\label{HoleFilling}
Let $G\in\G(1,n)$ and let $K'$ be a subgraph of $G^\dagger$. Suppose that $c$ is a simple cycle  in $K'\cap G$ with $E(K'\cap int(c)) =\emptyset$. 
If $K'$ is $(3,6)$-tight then $K'\cup int(c)$ is $(3,6)$-tight.
\end{lem}

\begin{lem}
    \label{lem:TightSubgraphTwo}
     Let $G\in \mathcal G(1,n)$. Suppose that \( K' \) is a \( (3,6) \)-tight subgraph of 
\( G^\dagger \) with \( B^\dagger \subset K' \) and let \( K = K' \cap G \). 
Label the face of $K$ corresponding to \( B^\dagger\) by \( B \)
and every other non-triangular face by \( H \).
Then,
\begin{enumerate}[(i)]
\item \( K \) is a face graph.
\item The boundary cycle of every $H$-labelled face in \( K \) is either 
    the boundary of a $H$-labelled face in \( G \) or is a 
    non-facial critical separating cycle in \( G \).
    \end{enumerate}
\end{lem}

\begin{proof}
    $(i)$ We need to show that the boundary cycle of each $H$-labelled face of $K$ is simple. If this were not the case then the boundary cycle of some face of $K$ would contain a repeated vertex. Note that this repeated vertex is a cut vertex for $K$. It is also a cut vertex for $K'$. However, \( K' \) does not have a cut vertex since it is \( (3,6) \)-tight. 
    
    $(ii)$ Suppose $c$ is the boundary cycle of a $H$-labelled face in $K$ which is not a $H$-labelled face in $G$. Let $G_1$ be the external face graph associated with $c$. Note that the external discus-and-hole graph  $G_1^\dagger$ is obtained from $K'$ by ``filling in" $H$-labelled faces of $K$. Since $K'$ is $(3,6)$-tight, by the hole-filling lemma (Lemma \ref{HoleFilling}), $G_1^\dagger$ is also $(3,6)$-tight. Thus, $c$ is a non-facial critical separating cycle in \( G \).  
\end{proof}

We will require the following result, known as the {\em isostatic substitution principle}. See \cite[Corollary 2.8]{whiteley} and the more general form \cite[Corollary 2.6]{frw}.

\begin{lem}
\label{l:isp}
Let $K$ be a simple graph which is minimally $3$-rigid and let $K'$ be a vertex induced subgraph of $K$ which is also minimally $3$-rigid. If $K'$ is replaced with another minimally $3$-rigid graph $K''$ with the property that $V(K')\subseteq V(K'')$ then the resulting graph is minimally $3$-rigid.
\end{lem}

\begin{lem}
\label{l:csc}
Let $G\in \mathcal{G}(1,n)$. Suppose $c$ is a non-facial critical separating cycle for $G$ with internal face graph $G_2$. If $d$ is a critical separating cycle for $G_2$ then $d$ is also a critical separating cycle for $G$.
\end{lem}

\proof
 By Theorem \ref{t:ckpthm}, the  discus-and-hole graphs $Ext_{G_2}(d)$ and $Ext_G(c)$ are minimally $3$-rigid. Note that  $Ext_G(d)$ is obtained  by replacing the discus in  $Ext_{G_2}(d)$ with $Ext_G(c)$. Thus, by the isostatic substitution principle (Lemma \ref{l:isp}), since $Ext_{G_2}(d)$ is minimally $3$-rigid,   
 $Ext_G(d)$ is also minimally $3$-rigid. 
 We conclude that $d$ is  a critical separating cycle for $G$.
\endproof

We now present a key technical lemma which is needed for the proof of Theorem
\ref{thm:MainImproved}  below.

\begin{figure}
\definecolor{zzttqq}{rgb}{0.6,0.2,0}
\definecolor{xdxdff}{rgb}{0.49,0.49,1}
\definecolor{qqqqff}{rgb}{0,0,1}
\begin{tikzpicture}[line cap=round,line join=round,x=1.0cm,y=1.0cm,scale=0.8]
\clip(-6.56,-1.1) rectangle (22.16,6.96);
\fill[color=zzttqq,fill=zzttqq,fill opacity=0.1] (-4,4.02) -- (-3.82,1.32) -- (-0.52,-0.14) -- (4.08,-1.06) -- (5.76,1.26) -- (5.98,3.82) -- (4.9,5.5) -- (2.3,6.4) -- (-0.92,6.42) -- (-3.32,6.12) -- cycle;
\draw(-0.8,3.48) circle (2.41cm);
\draw(2.62,2.64) circle (2.22cm);
\draw [shift={(2.62,2.64)},line width=2pt]  plot[domain=2.59:2.88,variable=\t]({1*2.22*cos(\t r)+0*2.22*sin(\t r)},{0*2.22*cos(\t r)+1*2.22*sin(\t r)});
\draw [shift={(2.62,2.64)},line width=2pt]  plot[domain=1.09:1.41,variable=\t]({1*2.22*cos(\t r)+0*2.22*sin(\t r)},{0*2.22*cos(\t r)+1*2.22*sin(\t r)});
\draw [color=zzttqq] (-4,4.02)-- (-3.82,1.32);
\draw [color=zzttqq] (-3.82,1.32)-- (-0.52,-0.14);
\draw [color=zzttqq] (-0.52,-0.14)-- (4.08,-1.06);
\draw [color=zzttqq] (4.08,-1.06)-- (5.76,1.26);
\draw [color=zzttqq] (5.76,1.26)-- (5.98,3.82);
\draw [color=zzttqq] (5.98,3.82)-- (4.9,5.5);
\draw [color=zzttqq] (4.9,5.5)-- (2.3,6.4);
\draw [color=zzttqq] (2.3,6.4)-- (-0.92,6.42);
\draw [color=zzttqq] (-0.92,6.42)-- (-3.32,6.12);
\draw [color=zzttqq] (-3.32,6.12)-- (-4,4.02);
\begin{scriptsize}
\fill [color=qqqqff] (0.48,3.22) circle (1.5pt);
\fill [color=xdxdff] (0.73,3.79) circle (1.5pt);
\fill [color=xdxdff] (2.96,4.83) circle (1.5pt);
\fill [color=xdxdff] (3.65,4.6) circle (1.5pt);
\draw[color=black] (0.38,3.54) node {$f$};
\draw[color=black] (3.54,4.86) node {$e$};
\draw[color=black] (2.24,0.76) node {$d$};
\draw[color=black] (-2.24,1.76) node {$c$};
\fill [color=qqqqff] (-4,4.02) circle (1.5pt);
\fill [color=qqqqff] (-3.82,1.32) circle (1.5pt);
\fill [color=qqqqff] (-0.52,-0.14) circle (1.5pt);
\fill [color=qqqqff] (4.08,-1.06) circle (1.5pt);
\fill [color=qqqqff] (5.76,1.26) circle (1.5pt);
\fill [color=qqqqff] (5.98,3.82) circle (1.5pt);
\fill [color=qqqqff] (4.9,5.5) circle (1.5pt);
\fill [color=qqqqff] (2.3,6.4) circle (1.5pt);
\fill [color=qqqqff] (-0.92,6.42) circle (1.5pt);
\fill [color=qqqqff] (-3.32,6.12) circle (1.5pt);
\end{scriptsize}
\end{tikzpicture}
\begin{caption}{Lemma \ref{lem:TightSubgraphThree}.}
    \label{fig:TightLemma}
\end{caption}

\end{figure}

\begin{lem}
    \label{lem:TightSubgraphThree}
    Let $G\in \mathcal G(1,n)$ and let \( c \) be a  critical separating cycle for \( G \) of length $|c|\geq 4$, with associated external and internal face graphs $G_1$ and $G_2$. Let $e$ be a $TT$ edge in $G_1$ and let $f$ be a $TT$ edge in $G_2$.
    \begin{enumerate}[(i)]
        \item If $e$ lies in a non-facial critical separating cycle for $G$ then $e$ also lies in a non-facial critical separating cycle for \( G_1 \).
        \item If $f$ lies in a non-facial critical separating cycle for $G$ then $f$ also lies in a non-facial critical separating cycle for \( G_2 \).
    \end{enumerate}
\end{lem}

\begin{proof}
    $(i)$ Suppose $d$ is a non-facial critical separating cycle for $G$ which contains the edge $e$ (see Figure \ref{fig:TightLemma} for an illustration).
    Let \( K' = Ext(c)\cap Ext(d) 
    \) and let \( K = K' \cap G \). 
    Similarly, let \( L' = Ext(c) \cup Ext(d) \) and let 
    \( L = L' \cap G \).
    Observe that,
    \[ f(K')+f(L') = f(Ext(c)) +f(Ext(d)) = 12.\]
    Therefore \( f(K') = f(L')= 6 \) and so $K'$ and $L'$ are  $(3,6)$-tight subgraphs of $G^\dagger$ which contain 
    \( B^\dagger\). 
    Label the face of $K$ corresponding to  \( B^\dagger \) by \( B \)
and every other non-triangular face of $K$ by \( H \). 
    Note that, since $|c|\geq 4$, \( e \) lies on the boundary cycle of a $H$-labelled face of \( K \) by construction. Let \( d' \) be this boundary cycle.
    Since \( e \) is a \( TT \) edge in \( G_1 \), \( d' \) cannot 
    be the boundary of a face in \( G_1 \).
    Therefore, by Lemma \ref{lem:TightSubgraphTwo},
    \( d'  \) is a non-facial critical separating cycle for \( G_1 \).
    This proves part   $(i)$. 
    Part $(ii)$ is proved by applying similar arguments to \( L \).
\end{proof}

\subsection{On indivisible graphs in \( \mathcal G(1,n) \)}
In this section, we derive  properties of face graphs in \( \mathcal G(1,n) \) which contain no $TT$ edges and no non-facial critical separating cycles. 

\begin{definition}
A face graph $G$ in $\G(1,n)$ is  \emph{indivisible} if every critical separating cycle for $G$ is the boundary cycle of a face of $G$.
\end{definition}

\begin{lem}
    \label{lem:LessThanTwoBHandIndivisibleImpliesContractibleTT}
    Suppose that \( G \in \mathcal G(1,n) \) has no  \( TT \) edge and is also indivisible. 
    Then \( G \) has at least three \( BH \) edges. 
\end{lem}

\begin{proof}

By \cite[Proposition 22(ii)]{ckp}, $G$ must contain at least one $BH$ edge. The cases where $G$ contains exactly one $BH$ edge and exactly two $BH$ edges are considered below. Since there are no $TT$ edges in $G$, for each vertex $v$ of $\partial B$ there exists a $H$-labelled face $H_v\in\H$ which contains $v$. The set of all $H$-labelled faces of $G$ is denoted by $\H$. Since $f(G^\dagger)=6$ it follows that $|\partial B| - 3 =\sum_{H\in\H}(|\partial H|-3)$.  
	
Case 1: Suppose $G$ contains exactly one $BH$ edge $e$. Then the vertices of $e$ are contained in a common $H$-labelled face $H_e$.
If the remaining $r=|\partial B|-2$ vertices $v_1,v_2,\ldots, v_r$ in $\partial B$ are each contained in distinct $H$-labelled faces then we obtain the contradiction, 
\[|\partial B| - 3 =\sum_{H\in\H}(|\partial H|-3) \geq (|\partial H_e|-3) + \sum_{i=1}^r (|\partial H_{v_i}|-3)\geq r+1.\] 

Case 2: Suppose $G$ contains exactly two $BH$ edges $e$ and $f$ and that these edges are adjacent.
The vertices of $e$ are contained in a common $H$-labelled face $H_e$. 
If the remaining $r=|\partial B|-3$ vertices $v_1,v_2,\ldots, v_r$ in $\partial B$ are each contained in  distinct $H$-labelled faces then we obtain the contradiction, 
 \[|\partial B| - 3 =\sum_{H\in\H}(|\partial H|-3) \geq (|\partial H_{e}|-3)  + \sum_{i=1}^r (|\partial H_{v_i}|-3)\geq r+1.\]

Case 3: Suppose $G$ contains exactly two $BH$ edges $e$ and $f$ and that these edges are not adjacent. The vertices of $e$ are contained in a common $H$-labelled face $H_e$ and  the vertices of $f$ are contained in a common $H$-labelled face $H_f$. 
If $H_e$ and $H_f$ are distinct, and, if the remaining $r=|\partial B|-4$ vertices $v_1,v_2,\ldots, v_r$ in $\partial B$ are each contained in distinct $H$-labelled faces then we obtain the contradiction, 
\[|\partial B| - 3 =\sum_{H\in\H}(|\partial H|-3) \geq (|\partial H_e|-3) + (|\partial H_f|-3) + \sum_{i=1}^r (|\partial H_{v_i}|-3)\geq r+2.\]

The contradictions obtained in each of the above cases imply that there must exist a pair of vertices $v$ and $w$ in $\partial B$ which are not joined by a $BH$-edge and for which $H_{v}=H_{w}$. By Lemma \ref{l:commonH}, there must exist a non-facial critical separating cycle in $G$.  However, this contradicts the indivisibility of $G$ and so $G$ must contain at least three $BH$ edges.  

\end{proof}

\begin{lem}
    \label{lem:ExactlyThreeBH}
    Suppose that \( G \in \mathcal G(1,n) \) has no \( TT \)
    edges, is indivisible and has exactly three \( BH \) edges.
    Then 
    \begin{enumerate}[(i)]
        \item Every $H$-labelled face in \( G \) is a quadrilateral.
        \item The three \( BH \) edges are not consecutive edges in $\partial B$.
    \end{enumerate}
\end{lem}

\begin{proof}
    Consider the following three cases.
		
		Case 1: Suppose $G$ contains exactly three $BH$ edges $e,f,g$ and no two are adjacent. 
		Then the vertices of $e,f,g$ are respectively contained in common $H$-labelled faces $H_e$, $H_f$ and $H_g$. 
		Since $G$ is indivisible, the faces $H_e$, $H_f$ and $H_g$ are distinct and the remaining $r=|\partial B|-6$ vertices $v_1,v_2,\ldots, v_r$ in $\partial B$ are each contained in a distinct $H$-labelled face. Thus, 
\begin{eqnarray*}
|\partial B| - 3 &=& \sum_{H\in\H}(|\partial H|-3)  \\
&\geq& (|\partial H_e|-3) +(|\partial H_f|-3)+(|\partial H_g|-3)+ \sum_{i=1}^r (|\partial H_{v_i}|-3) \\
&\geq& r+3
\end{eqnarray*} 
The above inequalities imply that $H_e$, $H_f$, $H_g$ and $H_{v_1},\ldots,H_{v_r}$ are the only $H$-labelled faces of $G$ and each of these faces has boundary length four. 

Case 2: Suppose $G$ contains exactly three $BH$ edges $e,f,g$ and exactly two of these edges, $e$ and $f$ say,  are adjacent. 
The vertices of $e$ and $g$ are respectively contained in common $H$-labelled faces $H_e$ and $H_g$. 
 Since $G$ is indivisible, the faces $H_e$ and $H_g$ are distinct and the remaining 
$r=|\partial B|-5$ vertices $v_1,v_2,\ldots, v_r$ in $\partial B$ are each contained in distinct $H$-labelled faces. Thus, 
\begin{eqnarray*}
|\partial B| - 3 &=&\sum_{H\in\H}(|\partial H|-3)  \\
&\geq& (|\partial H_e|-3) +(|\partial H_g|-3)+\sum_{i=1}^r (|\partial H_{v_i}|-3)\\
&\geq& r+2
\end{eqnarray*} 
The above inequalities imply that $H_e$, $H_g$ and $H_{v_1},\ldots,H_{v_r}$ are the only $H$-labelled faces of $G$ and each of these faces has boundary length four. 

Case 3: Suppose $G$ contains exactly three $BH$ edges $e,f,g$ and these three edges are consecutive. 
The vertices of $e$ are contained in a common $H$-labelled face $H_e$. 
 Since $G$ is indivisible, it follows from Lemma \ref{l:commonH} that the remaining 
$r=|\partial B|-4$ vertices $v_1,v_2,\ldots, v_r$ in $\partial B$ are each contained in distinct $H$-labelled faces. Thus, 
\begin{eqnarray*}
|\partial B| - 3 &=&\sum_{H\in\H}(|\partial H|-3)  \\
&\geq& (|\partial H_e|-3) +\sum_{i=1}^r (|\partial H_{v_i}|-3)\\
&\geq& r+1
\end{eqnarray*} 
The above inequalities imply that $H_e$  and $H_{v_1},\ldots,H_{v_r}$ are the only $H$-labelled faces of $G$ and each of these faces has boundary length four. 		
However,  the boundary of $H_e$ consists of three consecutive edges of $\partial B$ and a fourth edge that is not 
    in $B^\dagger$ but is incident to two vertices of $B^\dagger$. This contradicts the \( (3,6) \)-tightness of \( G^\dagger \) and so the three $BH$-edges of $G$ must not be consecutive. 
\end{proof}

See Figure \ref{fig:TerminalIndivisibleThreeBH}
for examples of  face graphs with no 
\( TT \) edges and exactly three \( BH \) edges. 

\begin{figure}
\definecolor{uququq}{rgb}{0.25,0.25,0.25}
\definecolor{zzttqq}{rgb}{0.6,0.2,0}
\definecolor{qqqqff}{rgb}{0,0,1}

\begin{tikzpicture}[line cap=round,line join=round,x=0.7cm,y=0.7cm]
\begin{scope}[shift={(0,3)},scale=0.5]
\clip(-7.1,-0.94) rectangle (17.52,8.88);
\draw [color=zzttqq] (-1.74,-0.08)-- (2.8,-0.08);
\draw [color=zzttqq] (2.8,-0.08)-- (5.07,3.85);
\draw [color=zzttqq] (5.07,3.85)-- (2.8,7.78);
\draw [color=zzttqq] (2.8,7.78)-- (-1.74,7.78);
\draw [color=zzttqq] (-1.74,7.78)-- (-4.01,3.85);
\draw [color=zzttqq] (-4.01,3.85)-- (-1.74,-0.08);
\draw (-1.74,7.78)-- (-1.66,5.22);
\draw (-1.66,5.22)-- (2.66,5.34);
\draw (2.66,5.34)-- (2.8,7.78);
\draw (2.66,5.34)-- (5.07,3.85);
\draw (2.66,5.34)-- (0.56,1.84);
\draw (0.56,1.84)-- (2.8,-0.08);
\draw (0.56,1.84)-- (-1.74,-0.08);
\draw (0.56,1.84)-- (-1.66,5.22);
\draw (-1.66,5.22)-- (-4.01,3.85);
\draw (-0.3,6.9) node[anchor=north west] {$ H $};
\draw (2.38,3.16) node[anchor=north west] {$ H $};
\draw (-2.22,3.2) node[anchor=north west] {$ H $};
\draw (-4.74,6.84) node[anchor=north west] {$ B $};
\begin{scriptsize}
\fill  (-1.74,-0.08) circle (2.5pt);
\fill  (2.8,-0.08) circle (2.5pt);
\fill [color=uququq] (5.07,3.85) circle (2.5pt);
\fill [color=uququq] (2.8,7.78) circle (2.5pt);
\fill [color=uququq] (-1.74,7.78) circle (2.5pt);
\fill [color=uququq] (-4.01,3.85) circle (2.5pt);
\fill  (-1.66,5.22) circle (2.5pt);
\fill  (2.66,5.34) circle (2.5pt);
\fill (0.56,1.84) circle (2.5pt);
\end{scriptsize}

\end{scope}

\begin{scope}[shift={(5.5,3)},scale=0.5]
\clip(-7.1,-0.94) rectangle (17.52,8.88);

  \coordinate (A1) at (-1.5,7.7);
  \coordinate (A2) at (2.5,7.7);
  \coordinate (A3) at (4.5,5.5);
  \coordinate (A4) at (4.5,2.5);
  \coordinate (A5) at (0.5,-0.08);
  \coordinate (A6) at (-3.5,2.5);
  \coordinate (A7) at (-3.5,5.5);
  \coordinate (A8) at (-1.5,5.5);
  \coordinate (A9) at (2.5,5.5);
  \coordinate (A10) at (2.5,2.5);
  \coordinate (A11) at (-1.5,2.5);
  
\draw [color=zzttqq] (A1) -- (A2) -- (A3) -- (A4) -- (A5) -- (A6) -- (A7) -- (A1);
\draw (A8) -- (A9) -- (A3);
\draw (A1)  -- (A8) -- (A11) -- (A5);
\draw (A2) -- (A9) -- (A10);
\draw  (A6) -- (A11);
\draw (A10) -- (A4);
\draw[color=blue] (A5) -- (A6) -- (A7) -- (A8) -- (A10) -- (A5);
\draw (-0.3,7.1) node[anchor=north west] {$ H $};
\draw (2.8,4.5) node[anchor=north west] {$ H $};
\draw (-3.2,4.5) node[anchor=north west] {$ H $};
\draw (-0.5,3.3) node[anchor=north west] {$ H $};
\draw (-5.1,6.84) node[anchor=north west] {$ B $};
\begin{scriptsize}
\fill [color=uququq] (A1) circle (2.5pt);
\fill [color=uququq] (A2) circle (2.5pt);
\fill [color=uququq] (A3) circle (2.5pt);
\fill [color=uququq] (A4) circle (2.5pt);
\fill [color=uququq] (A5) circle (2.5pt);
\fill [color=uququq] (A6) circle (2.5pt);
\fill [color=uququq] (A7) circle (2.5pt);
\fill [color=uququq] (A8) circle (2.5pt);
\fill [color=uququq] (A9) circle (2.5pt);
\fill [color=uququq] (A10) circle (2.5pt);
\fill [color=uququq] (A11) circle (2.5pt);

\end{scriptsize}

\end{scope}

\begin{scope}[shift={(11,3)},scale=0.5]
\clip(-7.1,-0.94) rectangle (17.52,8.88);

  \coordinate (A1) at (-1.5,7.7);
  \coordinate (A2) at (2.5,7.7);
  \coordinate (A3) at (4.5,5.5);
  \coordinate (A4) at (4.5,2.5);
  \coordinate (A5) at (2.5,0);
  \coordinate (A6) at (-1.5,0);
  \coordinate (A7) at (-3.5,2.5);
  \coordinate (A8) at (-3.5,5.5);
  \coordinate (A9) at (-1.5,5.5);
  \coordinate (A10) at (2.5,5.5);
  \coordinate (A11) at (2.5,2.5);
  \coordinate (A12) at (-1.5,2.5);
  \coordinate (A13) at (0.5,2.5);
  
\draw [color=zzttqq] (A1) -- (A2) -- (A3) -- (A4) -- (A5) -- (A6) -- (A7) -- (A8) -- (A1);
\draw (A9) -- (A10) -- (A3);
\draw (A7) -- (A12) -- (A9) -- (A1);
\draw (A4) -- (A11) -- (A10) -- (A2);
\draw (A12) -- (A6);
\draw (A13) -- (A10);
\draw (A13) -- (A5) -- (A11);
\draw[color=blue] (A6) -- (A7) -- (A8) -- (A9) -- (A13) -- (A6);

\draw (-0.3,7.1) node[anchor=north west] {$ H $};
\draw (2.8,4.5) node[anchor=north west] {$ H $};
\draw (-3.2,4.5) node[anchor=north west] {$ H $};
\draw (0.9,3.4) node[anchor=north west] {$ H $};
\draw (-1.5,3.4) node[anchor=north west] {$ H $};
\draw (-5.1,6.84) node[anchor=north west] {$ B $};
\begin{scriptsize}
\fill [color=uququq] (A1) circle (2.5pt);
\fill [color=uququq] (A2) circle (2.5pt);
\fill [color=uququq] (A3) circle (2.5pt);
\fill [color=uququq] (A4) circle (2.5pt);
\fill [color=uququq] (A5) circle (2.5pt);
\fill [color=uququq] (A6) circle (2.5pt);
\fill [color=uququq] (A7) circle (2.5pt);
\fill [color=uququq] (A8) circle (2.5pt);
\fill [color=uququq] (A9) circle (2.5pt);
\fill [color=uququq] (A10) circle (2.5pt);
\fill [color=uququq] (A11) circle (2.5pt);
\fill [color=uququq] (A12) circle (2.5pt);
\fill [color=uququq] (A13) circle (2.5pt);
\end{scriptsize}

\end{scope}

\end{tikzpicture}
\begin{caption}{Face graphs with no \( TT \) edges and exactly three \( BH \) edges. The face graph on the left lies in $\mathcal G(1,3)$ and is indivisible. The face graphs in the middle and on the right lie in $\mathcal G(1,4)$ and $\mathcal G(1,5)$ respectively and contain non-facial critical separating cycles (indicated in blue). }
    \label{fig:TerminalIndivisibleThreeBH}
\end{caption}
\end{figure}

\subsection{On the sufficiency of vertex splitting}
Let $G\in \mathcal{G}(1,n)$.
 A $TT$ edge is {\em contractible} in $G$ if it does not belong to any non-facial $3$-cycle in $G$.
A {\em $TT$ edge contraction} on $G$ is an operation on the class of face graphs  whereby the vertices of a contractible $TT$ edge in $G$ are identified, the resulting loop and parallel edges are discarded, and the labellings of all non-triangular faces in the resulting planar graph are inherited from $G$.
Note that a $TT$ edge contraction fails to preserve $(3,6)$-tightness if and only if the contractible $TT$ edge lies on a non-facial critical separating cycle of $G$ (see \cite[Lemma 27]{ckp}). For this reason we restrict attention to $TT$ edge contractions on $G$ which are {\em admissible} in the sense that the contractible $TT$ edge does not belong to a non-facial critical separating cycle of $G$. 

\begin{definition}
A face graph $G\in \mathcal G(1,n)$ is {\em terminal} if there exist no admissible $TT$ edge contractions on $G$.
\end{definition}

\begin{lem}
\label{l:terminal}
Let $G\in \mathcal{G}(1,n)$. 
If $G$ is terminal then $G$ contains no non-facial $3$-cycles.
\end{lem}

\proof
 Suppose $c$ is a non-facial $3$-cycle in $G$. Note that $f(c)=6$. Since $G^\dagger = Ext(c) \cup G_2$ and $c=Ext(c)\cap G_2$ we have,
    \[f(G_2) =f(Ext(c))+f(G_2)-f(c) =f(G^\dagger) =6.\]
    Recall that in general planar graphs satisfy $f(K)\geq 6$ and so $G_2$ is a maximal planar graph. Since $c$ is a non-facial $3$-cycle in $G$ it follows that there exists a contractible $TT$ edge $f$ in $G_2$ that does not lie in $c$ (see for example \cite[Lemma 1]{barnette}). 
    Note that the graph $G_2/f$ obtained on contracting this $TT$ edge is again a maximal planar graph. 
     Consider the face graph $G/f$ obtained from $G$ by applying a $TT$ edge contraction to $f$.  
     Note that the discus-and-hole graph $(G/f)^\dagger$ is obtained from $G^\dagger$ by replacing $G_2$ with $G_2/f$. Also note that, $G^\dagger$, $G_2$ and $G_2/f$ are minimally $3$-rigid.
     Thus, by the isostatic substitution principle (Lemma \ref{l:isp}), $(G/f)^\dagger$ is minimally $3$-rigid. 
     In particular, $(G/f)^\dagger$ is $(3,6)$-tight.
    Since the $TT$ edge contraction of $f$ preserves $(3,6)$-tightness it is an admissible $TT$ edge contraction on $G$. This contradicts the terminality of $G$. 
\endproof

A $BH$ edge in the face graph $G$ is {\em contractible} if it does not belong to any $3$-cycle in $G$.
A {\em $BH$ edge contraction} on $G$ is an operation on the class of face graphs whereby the vertices of a contractible $BH$ edge in $G$ are identified, the resulting loop is discarded, and the labellings of all non-triangular faces are inherited from $G$.   
Note that $BH$ edge contractions preserve $(3,6)$-tightness (see \cite[Lemma 29]{ckp}).
Also note that under a $BH$ edge contraction it is possible for the $B$-labelled face and the $H$-labelled face containing the contractible $BH$ edge to be transformed into triangular faces.

\begin{definition}
A face graph  is {\em $BH$-reduced} if it contains no contractible $BH$ edges. 
\end{definition}

We will require the following result.

\begin{lem}{\cite[Corollary 33]{ckp}}
\label{l:cor33}
For each $n\geq 1$, there is no face graph in $\G(1,n)$ which is terminal, indivisible and $BH$-reduced.
\end{lem}

Note that the reversal of a $TT$ edge contraction or a $BH$ edge contraction is a vertex splitting operation.
We can now strengthen the statement of Theorem \ref{t:ckpthm} as follows.

\begin{thm}
    \label{thm:MainImproved}
    Let \( \hat G \) be a block-and-hole graph with a single block and finitely many holes, or, a single hole and finitely many blocks. 
    The following statements are equivalent.
    \begin{enumerate}[(i)]
        \item \( \hat G \) is minimally \( 3 \)-rigid.
        \item \( G^\dagger \) is constructible from 
            \( K_3 \) by vertex splitting.
    \end{enumerate}
 \end{thm}

\begin{proof}
    Throughout this proof we will use the word ``constructible'' as a 
    shorthand for ``constructible from \( K_3 \) by vertex splitting only''.
    In light of  Theorem \ref{t:ckpthm} it suffices to show that if the discus-and-hole graph \( G^\dagger \) with a single discus and finitely many holes 
    is \( (3,6) \)-tight then it is constructible.
    We prove this by induction on the 
    number of edges in \( G^\dagger \). Thus let \( G \in \mathcal G(1,n) \)
    and assume that the theorem is true for all discus-and-hole graphs with strictly fewer edges than $G^\dagger$.
    If \( G \) has a contractible \( BH \) edge then by \cite[Lemma 29]{ckp} we can apply a $BH$ edge contraction to obtain a face graph $G'$ that lies in \( \mathcal G(1,n) \), \( \mathcal G(1,n-1)  \) or in \( \mathcal G(0,0) \). In any case, the resulting discus-and-hole graph $(G')^\dagger$ has fewer edges than $G^\dagger$ and is hence constructible. Note that $G^\dagger$ can be obtained from $(G')^\dagger$ by applying a vertex splitting operation and so \( G^\dagger \) is also constructible.
    Similarly, if \( G \) has a contractible \( TT \) edge that does not lie in any
    non-facial critical separating cycle then we may apply an admissible $TT$ edge contraction to obtain a face graph $G'$ which lies in \( \mathcal G(1,n) \).
    Again, the resulting discus-and-hole graph $(G')^\dagger$ has fewer edges than $G^\dagger$ and is hence constructible. Since  $G^\dagger$  can be obtained from $(G')^\dagger$ by vertex splitting we conclude that \( G^\dagger \) is  constructible also. 

    Now suppose \( G \) is both \( BH \)-reduced and terminal. 
    By Lemma \ref{l:terminal}, $G$ contains no non-facial $3$-cycles. Thus, by \ref{l:cor33}, 
    $G$ must contain a non-facial critical separating cycle \( c \) with $|c|\geq4$.    Let $G_1$ and $G_2$ be the external and internal face graphs associated with $c$.
    We can choose \( c \) so that there is no non-facial critical separating cycle 
    for \( G \) in \( G_2 \) apart from 
    \( c \) itself. 
    By Lemma \ref{l:csc}, any critical separating cycle for the internal face graph $G_2$ is also a critical separating cycle for $G$.  
    Thus, our choice of \( c \) ensures that the face graph \( G_2 \) is indivisible.

    If 
    \( G_2 \) contains a \( TT \) edge \( e \), then \( e \) does not lie on any 
    non-facial critical separating cycle of \( G_2 \). Since $|c| \geq 4$, $e\not\in c$ and so $e$ is also a $TT$ edge in $G$.
    By Lemma \ref{lem:TightSubgraphThree}, we conclude
    that \( e \) does not lie on any 
    non-facial critical separating cycle 
    for \( G \) either. Thus the contraction of $e$ is an admissible $TT$ edge contraction for $G$. This contradicts the terminality of \( G \) and so, from now on, we may assume that \( G_2 \) has no \( TT \) edges.     
    
    Suppose \( G_1 \) has a contractible \( TT \) edge \( e \)
    that does not lie  on any non-facial critical separating cycle of $G_1$. Since $|c| \geq 4$, $e$ is also a $TT$ edge in $G$. By Lemma \ref{lem:TightSubgraphThree}, $e$ does not lie on any non-facial critical separating cycle 
    of \( G \). Again, the contraction of $e$ is an admissible $TT$ edge contraction for $G$ and this contradicts the assumption that \( G \) is terminal. Thus, we may assume that $G_1$ is terminal.

     Since \( Ext(c) \) has fewer edges than $G^\dagger$, it is constructible. 
     Thus \( G_1 \) must have at least one contractible \( BH \) edge. 
    Since \( G \) is \( BH \)-reduced  and contains no non-facial $3$-cycles, we conclude that $G$ contains no $BH$ edges. Thus every  
    contractible \( BH \) edge of \( G_1 \) must in fact also be an edge of 
    \( c \) (otherwise it would be a \( BH \) edge in \( G \)).

    \begin{claim}
        \label{cl:FourEdgesSurvive}
        There are at least four edges of \( c \) that are not in 
        the boundary of the $B$-labelled face in $G$.
    \end{claim}

    \begin{proof}[Proof of Claim.]
    Using the isostatic substitution principle (Lemma \ref{l:isp}), observe that $G_2^\dagger$ is $(3,6)$-tight since it is obtained from $G^\dagger$ by replacing $Ext(c)$ with a discus.
     Since \( G_2 \) is indivisible 
    and has no \( TT \) edges we can apply Lemma \ref{lem:LessThanTwoBHandIndivisibleImpliesContractibleTT}
    to conclude that \( G_2 \) has at least three \( BH \) edges. 
    None of these edges are contained in the boundary of the $B$-labelled face in $G$ 
    since $G$ contains no $BH$ edges. Thus, we have demonstrated the 
    existence of three of the required four edges. To get the
    fourth edge we use 
    Lemma \ref{lem:ExactlyThreeBH}.
    This says that in the case where \( G_2 \) has exactly three
    \( BH \) edges, these three edges are not consecutive around the
    boundary of the $B$-labelled face of \( G_2 \). Label these three edges
    \( e_1 \), \( e_2 \) and \( e_3 \).
    Now suppose that all other edges of \( c \) also belong to
    the boundary of the the $B$-labelled face in $G$.
    Since 
    \( e_1 \), \( e_2 \) and \( e_3 \) are not consecutive 
    in the cycle \( c \), at least one of these edges, say $e_1$ after relabelling if necessary, is not adjacent to either of the other two. Then the vertices of $e_1$ must lie in the boundary of the $B$-labelled face in $G$. It follows that $e_1$ is an edge of $G^\dagger$
    that is not in the discus \( B^\dagger \) but is incident 
    with two vertices in \( B^\dagger \). This 
    contradicts the \( (3,6) \)-tightness of \( G^\dagger \).
    \end{proof}

    Now let \( K \) be the face 
    graph obtained by applying $BH$ edge contractions to $G_1$ until no further $BH$ edge contractions are possible (recalling that all of these $BH$ edges lie in \( c \)). By Claim \ref{cl:FourEdgesSurvive} 
    there are at least four edges remaining in the cycle 
    corresponding to  \( c \). So this cycle still 
    bounds a hole in \( K \). Thus every \( TT \) edge of \( K \) is 
    also a \( TT \) edge of \( G_1 \). Moreover it is clear that 
    there is an obvious correspondence between the 
    non-facial critical separating cycles of \( K \) and those of 
    \( G_1 \), and,  that if a \( TT \) edge of \( K \) lies on 
    a non-facial critical separating cycle in \( K \) then it 
    does so in \( G_1 \). By induction \( K^\dagger \) is constructible
    and so $K$ must have a contractible \( TT \)  edge that does not lie on a 
    non-facial critical separating cycle (it has no contractible \( BH \) edges
    by construction). But this contradicts the assumption 
    that \( G_1 \) has no such edges.

    We conclude that $G$ cannot be both $BH$-reduced and terminal. This completes the proof.
\end{proof}

\section{$(3,0)$-sparsity and pebble games}
\label{s:(3,0)}
The main result of \cite{ckp} characterises minimal $3$-rigidity for block-and-hole graphs with a single block in terms of $(3,6)$-sparsity.
The aim of this section is to show that $(3,6)$-sparsity is equivalent to an a priori weaker sparsity condition on two related multigraphs.
The advantage of these characterisations is that they can be quickly
checked via a pebble game algorithm in the sense of \cite{leestreinu}, whereas the $(3,6)$-sparsity
condition lies outside the ``matroidal'' range and cannot be so
easily checked.  

Let $G$ be a face graph with a single $B$-labelled face. 
We denote by $G^{2\sigma}$ the multigraph constructed from the face graph $G$ by adjoining two self-loops to each vertex $v\in V(\partial B)$. 
 Let $G^- = G\setminus E(\partial B)$ be the graph obtained by
removing the edges in the boundary cycle $\partial B$ from $G$. We denote by $(G^-)^{3\sigma}$  the graph obtained from $G^-$ by adding three self-loops to each of the vertices of $\partial B$. We refer to $G^{2\sigma}$ and $(G^-)^{3\sigma}$ as looped face graphs.

A multigraph $J$ is said to be  {\em $(3,0)$-sparse} if $f(J')\geq 0$ for any subgraph $J'$. A multigraph $J$ is {\em $(3,0)$-tight} if it is $(3,0)$-sparse and $f(J)=0$.  For more on $(k,l)$-sparsity generally see \cite{leestreinu}. We will require the following lemma.

\begin{lem}
\label{l:3orientation}
A multigraph is $(3,0)$-tight if and only if there exists an outdegree 3 orientation of the edges of the multigraph.    
\end{lem}

\begin{proof}
Apply \cite[Theorem 8 and Lemma 10]{leestreinu}.
\end{proof}
    
\begin{figure}
\definecolor{uququq}{rgb}{0.25,0.25,0.25}
\definecolor{zzttqq}{rgb}{0.6,0.2,0}
\definecolor{qqqqff}{rgb}{0,0,1}

\begin{tikzpicture}[line cap=round,line join=round,x=0.7cm,y=0.7cm]
\begin{scope}[shift={(0,3)},scale=0.5]
\clip(-7.1,-0.94) rectangle (17.52,8.88);
\draw [color=zzttqq] (-1.74,-0.08)-- (2.8,-0.08);
\draw [color=zzttqq] (2.8,-0.08)-- (5.07,3.85);
\draw [color=zzttqq] (5.07,3.85)-- (2.8,7.78);
\draw [color=zzttqq] (2.8,7.78)-- (-1.74,7.78);
\draw [color=zzttqq] (-1.74,7.78)-- (-4.01,3.85);
\draw [color=zzttqq] (-4.01,3.85)-- (-1.74,-0.08);
\draw (-1.74,7.78)-- (-1.66,5.22);
\draw (-1.66,5.22)-- (2.66,5.34);
\draw (2.66,5.34)-- (2.8,7.78);
\draw (2.66,5.34)-- (5.07,3.85);
\draw (2.66,5.34)-- (0.56,1.84);
\draw (0.56,1.84)-- (2.8,-0.08);
\draw (0.56,1.84)-- (-1.74,-0.08);
\draw (0.56,1.84)-- (-1.66,5.22);
\draw (-1.66,5.22)-- (-4.01,3.85);
\draw (-0.3,6.9) node[anchor=north west] {$ H $};
\draw (2.38,3.16) node[anchor=north west] {$ H $};
\draw (-2.22,3.2) node[anchor=north west] {$ H $};
\draw (-4.74,6.84) node[anchor=north west] {$ B $};
\begin{scriptsize}
\fill [color=uququq] (-1.74,-0.08) circle (2.5pt);
\fill [color=uququq] (2.8,-0.08) circle (2.5pt);
\fill [color=uququq] (5.07,3.85) circle (2.5pt);
\fill [color=uququq] (2.8,7.78) circle (2.5pt);
\fill [color=uququq] (-1.74,7.78) circle (2.5pt);
\fill [color=uququq] (-4.01,3.85) circle (2.5pt);
\fill [color=uququq] (-1.66,5.22) circle (2.5pt);
\fill [color=uququq] (2.66,5.34) circle (2.5pt);
\fill [color=uququq] (0.56,1.84) circle (2.5pt);
\end{scriptsize}

\end{scope}

\begin{scope}[shift={(5.65,3.7)},scale=0.35]
\clip(-7.1,-2.7) rectangle (17.52,10.5);

\draw[-stealth]  [color=zzttqq] (-1.74,-0.08)-- (2.8,-0.08);
\draw[-stealth] [color=zzttqq] (2.8,-0.08)-- (5.07,3.85);
\draw[-stealth] [color=zzttqq] (5.07,3.85)-- (2.8,7.78);
\draw[-stealth] [color=zzttqq] (2.8,7.78)-- (-1.74,7.78);
\draw[-stealth] [color=zzttqq] (-1.74,7.78)-- (-4.01,3.85);
\draw[-stealth] [color=zzttqq] (-4.01,3.85)-- (-1.74,-0.08);
\draw[stealth-] (-1.74,7.78)-- (-1.66,5.22);
\draw[-stealth] (-1.66,5.22)-- (2.66,5.34);
\draw[-stealth] (2.66,5.34)-- (2.8,7.78);
\draw[-stealth] (2.66,5.34)-- (5.07,3.85);
\draw[-stealth] (2.66,5.34)-- (0.56,1.84);
\draw[-stealth] (0.56,1.84)-- (2.8,-0.08);
\draw[-stealth] (0.56,1.84)-- (-1.74,-0.08);
\draw[-stealth][-stealth] (0.56,1.84)-- (-1.66,5.22);
\draw[-stealth] (-1.66,5.22)-- (-4.01,3.85);

\begin{pgfinterruptboundingbox}

\tikzset{every loop/.style={min distance=25mm, looseness=15}}

\path (-1.74,-0.08) edge[loop below]  (-1.74,-0.08);
\path (2.8,-0.08) edge[loop below]  (2.8,-0.08);
\path (5.07,3.85) edge[in=30,out=60,loop]  (5.07,3.85);
\path (2.8,7.78) edge[loop above]  (2.8,7.78);
\path (-1.74,7.78) edge[loop above]  (-1.74,7.78);
\path (-4.01,3.85) edge[in=120,out=150,loop]  (-4.01,3.85);

\path (-1.74,-0.08) edge[loop left]  (-1.74,-0.08);
\path (2.8,-0.08) edge[loop right]  (2.8,-0.08);
\path (5.07,3.85) edge[in=-30,out=-60,loop]  (5.07,3.85);
\path (2.8,7.78) edge[loop right]  (2.8,7.78);
\path (-1.74,7.78) edge[loop left]  (-1.74,7.78);
\path (-4.01,3.85) edge[in=210,out=240,loop]  (-4.01,3.85);

  \end{pgfinterruptboundingbox}
  
\begin{scriptsize}
\fill [color=uququq] (-1.74,-0.08) circle (2.5pt);
\fill [color=uququq] (2.8,-0.08) circle (2.5pt);
\fill [color=uququq] (5.07,3.85) circle (2.5pt);
\fill [color=uququq] (2.8,7.78) circle (2.5pt);
\fill [color=uququq] (-1.74,7.78) circle (2.5pt);
\fill [color=uququq] (-4.01,3.85) circle (2.5pt);
\fill [color=uququq] (-1.66,5.22) circle (2.5pt);
\fill [color=uququq] (2.66,5.34) circle (2.5pt);
\fill [color=uququq] (0.56,1.84) circle (2.5pt);
\end{scriptsize}

\end{scope}

\begin{scope}[shift={(11,3.7)},scale=0.35]
\clip(-7.1,-2.7) rectangle (17.52,10.5);
\draw[stealth-] (-1.74,7.78)-- (-1.66,5.22);
\draw[-stealth] (-1.66,5.22)-- (2.66,5.34);
\draw[-stealth] (2.66,5.34)-- (2.8,7.78);
\draw[-stealth] (2.66,5.34)-- (5.07,3.85);
\draw[-stealth] (2.66,5.34)-- (0.56,1.84);
\draw[-stealth] (0.56,1.84)-- (2.8,-0.08);
\draw[-stealth] (0.56,1.84)-- (-1.74,-0.08);
\draw[-stealth][-stealth] (0.56,1.84)-- (-1.66,5.22);
\draw[-stealth] (-1.66,5.22)-- (-4.01,3.85);

\begin{pgfinterruptboundingbox}

\tikzset{every loop/.style={min distance=25mm, looseness=15}}

\path (-1.74,-0.08) edge[loop below]  (-1.74,-0.08);
\path (2.8,-0.08) edge[loop below]  (2.8,-0.08);
\path (5.07,3.85) edge[in=30,out=60,loop]  (5.07,3.85);
\path (2.8,7.78) edge[loop above]  (2.8,7.78);
\path (-1.74,7.78) edge[loop above]  (-1.74,7.78);
\path (-4.01,3.85) edge[in=120,out=150,loop]  (-4.01,3.85);

\path (-1.74,-0.08) edge[loop left]  (-1.74,-0.08);
\path (2.8,-0.08) edge[loop right]  (2.8,-0.08);
\path (5.07,3.85) edge[in=-30,out=-60,loop]  (5.07,3.85);
\path (2.8,7.78) edge[loop right]  (2.8,7.78);
\path (-1.74,7.78) edge[loop left]  (-1.74,7.78);
\path (-4.01,3.85) edge[in=210,out=240,loop]  (-4.01,3.85);

\path (-1.74,-0.08) edge[in=210,out=240,loop]  (-1.74,-0.08);
\path (2.8,-0.08) edge[in=-30,out=-60,loop]  (2.8,-0.08);
\path (5.07,3.85) edge[loop right]  (5.07,3.85);
\path (2.8,7.78) edge[in=30,out=60,loop]  (2.8,7.78);
\path (-1.74,7.78) edge[in=120,out=150,loop]  (-1.74,7.78);
\path (-4.01,3.85) edge[loop left]  (-4.01,3.85);

  \end{pgfinterruptboundingbox}
  
\begin{scriptsize}
\fill [color=uququq] (-1.74,-0.08) circle (2.5pt);
\fill [color=uququq] (2.8,-0.08) circle (2.5pt);
\fill [color=uququq] (5.07,3.85) circle (2.5pt);
\fill [color=uququq] (2.8,7.78) circle (2.5pt);
\fill [color=uququq] (-1.74,7.78) circle (2.5pt);
\fill [color=uququq] (-4.01,3.85) circle (2.5pt);
\fill [color=uququq] (-1.66,5.22) circle (2.5pt);
\fill [color=uququq] (2.66,5.34) circle (2.5pt);
\fill [color=uququq] (0.56,1.84) circle (2.5pt);
\end{scriptsize}

\end{scope}

\end{tikzpicture}
\begin{caption}{A face graph $G$ (left) and its associated looped face graphs $G^{2\sigma}$ (centre) and $(G^-)^{3\sigma}$ (right) together with out degree $3$ edge orientations.} 
    \label{fig:looped}
\end{caption}
\end{figure}

\begin{example}
Let $\hat G$ be a block-and hole graph on the face graph $G$ illustrated in Figure \ref{fig:looped}. The associated looped face graphs admit out degree $3$ edge orientations. Thus, by Lemma \ref{l:3orientation}, these multigraphs are $(3,0)$-tight. By Theorem \ref{thm:(3,0)} below, the block and hole graph $\hat G$ is $(3,6)$-tight and so, by Theorem \ref{t:ckpthm}, $\hat G$ is minimally $3$-rigid.
\end{example}

We now prove the main result of this section.

\begin{thm}\label{thm:(3,0)}
Let $\hat{G}$ be a block-and-hole graph with a single block and finitely many holes.
Then the following statements are equivalent.
\begin{enumerate}[(i)]
\item $\hat{G}$ is minimally $3$-rigid. 
\item $G^{2\sigma}$ is $(3,0)$-tight.
\item $(G^-)^{3\sigma}$ is  $(3,0)$-tight.
\end{enumerate}
\end{thm}

\begin{proof}
$(i)\Rightarrow(ii)$
Suppose $\hat{G}$ is minimally $3$-rigid. 
Let $K$ be a  subgraph of
$G^{2\sigma}$ and let $K'=K\cap G$ be the subgraph of $G$ obtained by removing all  self-loops from $K$. 
Note that $K'\cap\hat{B}$ is a subgraph of the boundary cycle  $\partial B$ and so $|E(K'\cap\hat{B})|\leq |V(K'\cap\hat{B})|$. It follows that $f(K'\cap\hat{B})\geq 2|V(K'\cap\partial B)|$.
Note that 
\[f(K'\cup \hat{B})=f(K')+f(\hat{B})-f(K'\cap\hat{B})
\leq f(K')+6-2|V(K'\cap \partial B)|
\leq f(K)+6.\]
Since $K'\cup\hat{B}$ is a subgraph of $\hat{G}$, it is
$(3,6)$-sparse, and so $f(K)\geq0$. 
We conclude that $G^{2\sigma}$ is $(3,0)$-sparse.
Note that $f(\partial B)=2|V(\partial B)|$ and so,  
\[f(\hat{G})=f(\hat{B})+f(G)-f(\partial B)
=6+f(G)- 2|V(\partial B)|=6+f(G^{2\sigma}).\]
Thus $f(G^{2\sigma})=f(\hat{G})-6=0$ and so $G^{2\sigma}$ is $(3,0)$-tight.

$(ii)\Leftrightarrow (iii)$
Note that on $V(\partial B)$, any outdegree $3$ orientation of the edges of $(G^-)^{3\sigma}$ or $G^{2\sigma}$ has a very constrained form.
For any vertex $v$ of $V(\partial B)\subset V((G^-)^{3\sigma})$, the three
self-loops on it must be oriented away from $v$, and similarly for the
two self-loops on the vertices of $\partial B\subset G^{2\sigma}$. Then
there is one remaining outgoing edge from each $v\in V(\partial B)\subset
V(G^{2\sigma})$ which must be one of the two edges of $\partial B$ that meet it. 
It follows that $\partial B\subset G^{2\sigma}$ must be oriented according to 
one of its two cyclic orientations.  Thus any outdegree 3 orientation of $(G^-)^{3\sigma}$ is
easily converted to one of $G^{2\sigma}$ and vice versa. The result now follows from Lemma \ref{l:3orientation}.

$(iii)\Rightarrow(i)$
Suppose  the multigraph $(G^-)^{3\sigma}$ is $(3,0)$-tight. 
Let $K$ be a subgraph of $\hat{G}$ containing at least two edges.
If $K$ is a subgraph of $G$ then, since $G$ is a
subgraph of a triangulated sphere, $K$ is $(3,6)$-sparse.  
If $K$ is not a subgraph of $G$ then we consider three possible cases: 

{\em Case 1:} Suppose $K\cap\hat{B}$ contains
at least two edges.
Consider the subgraph $(K\cap G^-)^{3\sigma}$ of the multigraph $(G^-)^{3\sigma}$.
Note that,
\[0\leq f((K\cap G^-)^{3\sigma})=f(K\cap G)-f(K\cap \partial B).
\]
Since $\hat{B}$ is $(3,6)$-sparse, we have $f(K\cap\hat{B})\geq 6$ and so,
\[f(K)=f(K\cap\hat{B})+f(K\cap G)-f(K\cap \partial B)\geq 6.
\]

{\em Case 2:} Suppose $K\cap\hat{B}$ contains no edges, or contains exactly one edge which lies in $\partial B$. Then $K$ must be the disjoint union 
of $K\cap G$ (which, as a subgraph of a triangulated sphere, is $(3,6)$-sparse) 
and some number of vertices in $\hat{B}$. Hence $f(K)\geq f(K\cap G) \geq 6$.

{\em Case 3:} Suppose $K\cap\hat{B}$ contains exactly one edge and that this edge does not lie in $\partial B$. Then $K$ must consist of $K\cap G$ with an
additional edge (which is still a subgraph of a triangulated sphere) 
together with some number of vertices in $\hat{B}$. Hence $f(K) \geq 6$.

We conclude that $\hat{G}$ is $(3,6)$-sparse.
Also, 
\[f(\hat{G})=f(\hat{B})+f(G)-f(\partial B)
=6+f(G)- 2|V(\partial B)|=6+f((G^-)^{3\sigma}).\]
Thus $f(\hat{G})=6$ and so $\hat{G}$ is $(3,6)$-tight.
By Theorem \ref{t:ckpthm}, $\hat{G}$ is minimally $3$-rigid.
\end{proof}

\section{Applications and Conjectures}
\label{s:applications}

\subsection{Rigidity in $\ell_p^3$}

The vertex splitting operation considered in Section \ref{s:split} is known to preserve  rigidity properties in geometric settings other than the Euclidean space $\mathbb{R}^3$. For example, it is known that vertex splitting preserves {\em independence} in every $3$-dimensional real normed linear space which is both smooth and strictly convex (see \cite[Proposition 4.7]{dkn}). It follows that any class of graphs which are constructible from an independent base graph by vertex splitting (for example, triangulations of a 2-sphere) will satisfy independence. Thus, with the main theorem of Section \ref{s:split} in hand, we obtain the following immediate corollary.

\begin{cor}
Let $X$ be a $3$-dimensional real normed linear space which is smooth and strictly convex.
Then every $(3,6)$-tight discus-and-hole graph, with a single discus, is independent in $X$.
\end{cor}

\proof
By \cite[Proposition 4.7]{dkn}, vertex splitting preserves independence in $X$. The graph $K_3$ is independent in $X$. Thus the result follows from Theorem \ref{thm:MainImproved}.
\endproof

In the case of $\ell_p^3$, where $p\in[1,\infty]$ and $p\not=2$, the minimally rigid graphs are $(3,3)$-tight. Here a simple graph $J$ is  {\em $(3,3)$-tight} if $f(J)=6$ and $f(J')\geq 3$ for any subgraph $J'$. The smallest (non-trivial) graph with this property is the complete graph $K_6$. It is conjectured that every $(3,3)$-tight simple graph is minimally rigid in  $\ell_p^3$ (see for example \cite{dkn}). We propose here a special case of this conjecture.

\begin{conjecture}
       Let $p\in[1,\infty]$, $p\not=2$. Let $\hat{G}$ be a block-and-hole graph with a single block. If the block is minimally rigid in $\ell_p^3$ 
       then the following statements are equivalent.
    \begin{enumerate}[(i)]
        \item $\hat{G}$ is minimally rigid in $\ell_p^3$.
        \item $\hat{G}$ is $(3,3)$-tight.
    \end{enumerate}
\end{conjecture}

\subsection{Conjecture on global rigidity}

Establishing global rigidity is typically a more difficult problem than 
establishing rigidity for a given class of graphs. One of the reasons is that 
vertex splitting is less well understood in this context. Connelly and Whiteley have conjectured a necessary and sufficient condition for vertex splitting to preserve global rigidity in $\mathbb{R}^d$ \cite{CW}. This conjecture is still open but 
has been verified in certain special cases (see \cite{JT,CJT1,CJT2}) leading to global rigidity characterisations for braced plane triangulations and for triangulations of non-spherical surfaces. Given Theorem \ref{thm:MainImproved}, it is natural to wonder if similar global rigidity 
characterisations might be obtained for discus-and-hole graphs. 

\begin{conjecture}
    \label{conj:GR}
    Suppose that $G^\dagger$ is a discus-and-hole graph with exactly one discus.
    Then $G^\dagger$
    is generically globally rigid in $\mathbb R^3$ if and only if $G^\dagger$
    is $4$-connected and redundantly rigid in $\mathbb R^3$. 
\end{conjecture}

Note that the ``only if" implication in Conjecture \ref{conj:GR} is already well known (see \cite{Hend}).

\subsection{Connection to rigid origami}
\label{s:origami}
Rigid origami is the study of structures made out of flat rigid sheets
joined at hinges.  Such structures have inspired work in structural
engineering, mechanical design and the physics of mechanical
metamaterials \cite{glaucio,guest,jesse,wei}.  It is of practical interest, given
such a structure, to determine its mechanical properties, and as a
very first step, one would like to know whether it is floppy or rigid.
It is natural, given the constraint that the sheets remain rigidly
flat, to mathematically model rigid origami by polyhedral surfaces
(with boundary).

The connection to the block-and-hole graphs considered in this article
is then as follows.  Given a polyhedral surface, we wish to replace it
by a bar-joint framework such that all vertices and edges of the
polyhedral surface become joints and bars, respectively.  In order for
the framework to have the same rigidity properties we must add
additional bars and joints to the non-triangular faces, as they could
otherwise bend and flex in the framework.  By the isostatic
substitution principle (Lemma \ref{l:isp}), this can be done without introducing dependencies in the
bars by adding any minimally 3-rigid graph on the vertices of the
planar face.

For example, the following two part construction works: first,
triangulate each of the non-triangular faces and second, for each
non-triangular face, create a new joint off the plane of the face with
bars to each of the vertices of that face.  Note that this replaces
the rigid face with a triangulated prism.

One can then naturally identify these with ``blocks" and the missing
faces as ``holes".  One important caveat is that the realizations of
block-and-hole graphs arising from the above construction are not
generic -- the blocks are bounded by sets of coplanar vertices.  It is
natural of course to conjecture (along the lines of the molecular
conjecture of Tay and Whiteley \cite{taywhiteley} proved by Katoh and Tanigawa
\cite{molecular}) that the
rigidity of generic polyhedral surfaces can indeed be predicted by the
rigidity of structures where the blocks are made more generic, but
this remains to be proven.

One further point is that the definition of rigid origami above allows
vertices to have discrete Gaussian curvature (i.e.\ the angles of the
faces around them may not sum to $2\pi$).  Such a structure could not be
folded from an ordinary sheet of paper.  It would be interesting to
consider the ``developable" rigid origami case (where all angle-sums
around vertices are $2\pi$), and this would require the consideration of
further non-genericities.  It may be that block-and-hole graphs
provide the appropriate counts for ``generic developable rigid origami"
as well.

Assuming a suitable ``molecular origami conjecture" holds, 
Theorem \ref{t:ckpthm} and Theorem~\ref{thm:(3,0)} give a way
of determining the rigidity or flexibility of rigid origami with
either (1) one non-triangular face and an arbitrary number of
non-triangular holes or (2) one non-triangular hole and an arbitrary
number of non-triangular faces (related by block-and-hole swapping).  
Note that ``pure" origami folded from a
single-sheet without allowing any cutting leads at the combinatorial
level to block-and-hole graphs which satisfy (2), with the exterior of
the paper viewed as a large hole.

\section{Acknowledgement}
This article is based on work initiated by the authors during the BIRS workshop on Advances in Combinatorial and Geometric Rigidity (15w5114).

\end{document}